\documentclass[11pt,reqno]{amsart}

\usepackage{amsmath}
\usepackage{amssymb}
\usepackage{graphicx}
\newtheorem{thm}{Theorem}
\newtheorem{cor}{Corollary}
\newtheorem{lem}{Lemma}

\theoremstyle{definition}

\newtheorem{rem}{Remark}
\newcommand{\be}{\begin{equation}}
\newcommand{\ee}{\end{equation}}
\newcommand\bes{\begin{eqnarray}}
 \newcommand\ees{\end{eqnarray}}
\newcommand{\bess}{\begin{eqnarray*}}
\newcommand{\eess}{\end{eqnarray*}}

 \numberwithin{equation}{section}

\begin{document}
\vbox{\vskip 4cm}
\title{ The Generalized Entropy Ergodic Theorem for Nonhomogeneous Markov Chains}
\thanks{This work was supported by the National Natural Science Foundation of China (11071104)
, the National Natural Science Foundation of Anhui Province (1408085MA04)
and Foundation of Anhui Educational Committee (KJ2012B117). \\
 \indent{$^\dag$ The corresponding author, his email is wgyang@ujs.edu.cn}
 \indent}
\maketitle
\begin{center}
{\sl   Zhongzhi Wang $^{1}$   Weiguo Yang $^2 $$^\dag$
 }\\
\

1.School of Mathematics {\&}\ Physics, Anhui University of Technology, Ma'anshan, 243002, China\\
2.Faculty of Science, Jiangsu University, Zhenjiang, 212013, China
\end{center}

\

 \indent {\bf Abstract:}
       { \small Let $(\xi_n)_{n=0}^\infty$ be a nonhomogeneous Markov chain taking values from finite state-space of
       $\mathbf{X}=\{1,2,\ldots,b\}$. In this paper, we will study the generalized entropy ergodic theorem with almost-everywhere and $\mathcal{L}_1$ convergence for nonhomogeneous Markov chains, which generalizes the corresponding classical results for the Markov chains.

        \

{\bf{Key words and phrases:}} Nonhomogeneous Markov chains; generalized entropy ergodic theorem; almost-everywhere convergence\\

  \indent{\bf{2010 MR Subject Classification :}} 60F15, 94A37.\\

  \

\section{\bf Introduction}

\

We begin with introducing some notations that will be used throughout the paper. Assume that $(\xi_n)_{n=0}^\infty$ is a sequence of arbitrary random variables taking values from a finite set of
$\mathbf{X} =\{1,2,\cdots,b\}$ and
$(\Omega,\mathcal{F},\mathbb{P})$ the underlying probability space.
For convenience, denote by $
 \xi_{m,n}$ the random vector of $(\xi_m,\cdots,\xi_{m+n})$ and
 $x_{m,n}=(x_m,\cdots,x_{m+n})$, a realization of $\xi_{m,n}$. Suppose the joint distribution of $\xi_{m,n}$
 is
\bes \mathbb{P}(\xi_{m,n}=x_{m,n})=p(x_m,\cdots,x_{m+n})=p(x_{m,n}),
\ \ x_k\in \mathbf{X},\  m\leq k\leq m+n.\ees Let
$(a_n)_{n=0}^\infty$ and $(\phi(n))_{n=0}^\infty$ be two sequences
of nonnegative integers such that $\phi(n)$ converges to infinite as $n\rightarrow\infty$. Let \bes
f_{a_n,\phi(n)}(\omega)=-\frac{1}{\phi(n)}\log
p(\xi_{a_n,\phi(n)}),\ees where $\log$ is the natural logarithm.
$f_{a_n,\phi(n)}(\omega)$ will be called generalized entropy density
of $\xi_{a_n,\phi(n)}$. If $a_n\equiv 0$ and $\phi(n)=n$,
$f_{a_n,\phi(n)}(\omega)$ will become the classical entropy density
of $\xi_{0,n}$ defined as follows
 \bes f_{0,n}(\omega)=-\frac{1}{n}\log p(\xi_{0,n}).\ees
If $(\xi_n)_{n=0}^\infty$ is a nonhomogeneous Markov chain
 taking values in finite state-space of $\mathbf{X}
=\{1,2,\cdots,b\}$ with the initial distribution
 \bes(\mu_0(1),\cdots,\mu_0(b)),\ees and the
transition matrices \bes P_n=(p_{n}(i,j))_{b\times b},\ \ i,j\in
\mathbf{X},\ n=1,2\cdots,\ees where
$p_n(i,j)=\mathbb{P}(\xi_n=j|\xi_{n-1}=i)$, then
 \bes
f_{a_n,\phi(n)}(\omega)=-\frac{1}{\phi(n)}\left\{\log
\mu_{a_n}(\xi_{a_n})+\sum_{k=a_n+1}^{a_n+\phi(n)} \log
p_k(\xi_{k-1},\xi_k)\right\},\ees
where $\mu_{a_n}(x)$ is the distribution of $\xi_{a_n}$.

\

The convergence of $f_{0,n}(\omega)$ to a constant in a sense
of $\mathcal{L}_1$ convergence, convergence in probability or $a.e.$
convergence, is called Shannon-McMillan-Breiman theorem or entropy
ergodic theorem or asymptotic equipartition property (AEP) respectively in information theory. Shannon [10] first estabilished the entropy
ergodic theorem for convergence in probability for stationary
ergodic information sources with finite alphabet. McMillan [9] and
Breiman [3] obtained, for finite stationary ergodic information
sources, the entropy ergodic theorem in $\mathcal{L}_1$ and $a.e.$
convergence, respectively. Chung [6] considered the case of
countable alphabet. The entropy ergodic theorem for general
stochastic processes can be found, for example, in Barron [2],
Kieffer [8], or  Algoet and Cover [1]. Yang [12] obtained entropy
ergodic theorem for a class of nonhomogeneous Markov chains, Yang
and Liu [13], the entropy ergodic theorem for a class of $m$th-order
nonhomogeneous Markov chains, Zhong, Yang and Liang [14], entropy
ergodic theorem for a class of asymptotic circular Markov chains.

\

The second term of Eq. (1.6) is actually delayed sums of random
variables, which was first introduced by Zygmund
[15] who used it to prove a Tauberian theorem of Hardy. Since then, a lot of work has been done to investigate the properties of delayed sums. For example, by using the limiting behavior of delayed sums,
Chow [4] found necessary and sufficient conditions for the Borel summability of i.i.d. random variables and simplified the
proofs of a number of well-known results such as the
Hsu-Robbins-Spitzer-Katz theorem. Lai [11] studied the analogues of
the law of the iterated logarithm for delayed sums of independent
random variables. Recently, Gut and Stradtm\"{u}ller [7] studied the
strong law of large numbers for delayed sums of random fields.

\

Let $(\xi_n)_{n=0}^\infty$ be a nonhomogeneous Markov chain with the transition matrices (1.5).
 Yang [12] showed that the classical entropy density
$f_{0,n}(\omega)$ of this Markov chain converges $a.e.$ to the entropy rate of a Markov
chain  under the condition that
$\lim_{n\rightarrow\infty}\frac{1}{n}\sum_{k=1}^{n}|p_k(i,j)-p(i,j)|=0$,
for all $i,j\in \mathbf{X}$, where $P=(p(i,j))_{b\times b}$ is an
irreducible transition matrix. In this paper, we will prove that the
generalized entropy density $f_{a_n,\phi(n)}(\omega)$ converges
$a.e.$ and $\mathcal{L}_1$ to this entropy rate under some mild
conditions, which is called the generalized entropy ergodic theorem.
The results of this paper generalize the results of those in [12].

\

To prove the main results, we first establish a strong limit theorem for the delayed sums of the functions of two
variables for nonhomogeneous Markov chains, then we obtain the strong limit theorems of the frequencies of occurrence of states and the ordered couples of states in the segment
$\xi_{a_n},...,\xi_{a_n+\phi(n)}$ for the Markov chains. At the end
we present the main results. We also prove that
$f_{a_n,\phi(n)}(\omega)$ are uniformly integrable for arbitrary
finite sequence of random variables.

\

The approach used in this paper is different from the one used in
some previous works([12],[13]), where the strong law of large
numbers for martingale is applied. As $f_{a_n,\phi(n)}(\omega)$ is the delayed sums of $\log p_k(\xi_{k-1},\xi_{k})$, the strong law of large numbers for martingale cannot be applied. The essence of the technique used in this paper is first to construct a one parameter class of random variables with means of 1, then, using Borel-Cantelli lemma, to prove the existence of $a.e.$ convergence of certain random variables.

\

The rest of this paper is organized as follows. In section 2, we first establish some preliminary results that will be used to prove our main results, and present the main results of this paper and their proofs  in section 3.

\

 \section{\bf Some lemmas}

Before proving the main results, we first begin with some lemmas.

\

\begin{lem} Suppose $(\xi_n)_{n=0}^\infty$ is a nonhomogeneous Markov chain taking values from a finite state-space of $\mathbf{X} =\{1,2,\cdots,b\}$ with the initial
distribution (1.4) and the transition matrices (1.5). Suppose
$(a_n)_{n=0}^\infty$ and $(\phi(n))_{n=0}^\infty$ are two sequences
of nonnegative integers such that $\phi(n)$ tends to infinity as $n\rightarrow\infty$. Let $(g_n(x,y))_{n=0}^\infty$ be a sequence
of real functions defined on $\mathbf{X}\times\mathbf{X}$.
 If for every $\varepsilon>0$
\bes\sum_{n=1}^\infty\exp[-\varepsilon \phi(n)]<\infty,\ees
and there exists a real number $ 0<\gamma<\infty$ such that
\bes\limsup_n\frac{1}{\phi(n)}\sum_{k=a_n+1}^{a_n+\phi(n)}E[|g_k(\xi_{k-1},\xi_k)|^2
e^{\gamma|g_k(\xi_{k-1},\xi_k)|}|\xi_{k-1}]=c(\gamma;\omega)<\infty\
\ a.e.,\ees
 then, we have
\bes\lim_n\frac{1}{\phi(n)}\sum_{k=a_n+1}^{a_n+\phi(n)}\{g_k(\xi_{k-1},\xi_k)-E[g_k(\xi_{k-1},\xi_k)|\xi_{k-1}]\}=0\
\ a.e.\ees \end{lem}
\begin{rem} Obviously, condition (2.1) in Lemma
1 can be easily satisfied. For example, let $\phi(n)=[n^\alpha]
(\alpha>0)$, where $[\cdot]$ is the usual greatest integer function,
then (2.1) holds. If $(g_n(x,y))_{n=1}^\infty$ are uniformly
bounded, then Eq. (2.2) holds.
\end{rem}
\begin{rem} Since $E[g_k(\xi_{k-1},\xi_{k})|\xi_{k-1}]=\sum_{j=1}^bg_k(\xi_{k-1},j)p_k(\xi_{k-1},j)  $, Eq. (2.3) can be rewritten as
\bes\lim_{n}\frac{1}{\phi(n)}\sum_{k=a_n+1}^{a_n+\phi(n)}\{
g_k(\xi_{k-1},\xi_{k})-
\sum_{j=1}^bg_k(\xi_{k-1},j)p_k(\xi_{k-1},j)\}=0 \ \ a.e.\ees
 \end{rem}

 \

\begin{proof} Let $s$ be a nonzero real number, define
$$\Lambda_{a_n,\phi(n)}(s,\omega)=\frac{\exp\{s \sum_{k=a_n+1}^{a_n+\phi(n)}g_k(\xi_{k-1},\xi_k)\}}{\prod_{k=a_n+1}^{a_n+\phi(n)}E[e^{s g_k(\xi_{k-1},\xi_k)}|\xi_{k-1}]},\ \ n=1,2,\cdots.$$
and note that
\begin{align}&E\Lambda_{a_n,\phi(n)}(s,\omega)\notag\\=&E[E[\Lambda_{a_n,\phi(n)}(s,\omega)|\xi_{0,a_n+\phi(n)-1}]]\notag\\
=&E\left[E[\Lambda_{a_n,\phi(n)-1}(s,\omega)\frac{e^{s g_{a_n+\phi(n)}(\xi_{a_n+\phi(n)-1},\xi_{a_n+\phi(n)})}}
{E[e^{sg_{a_n+\phi(n)}(\xi_{a_n+\phi(n)-1},\xi_{a_n+\phi(n)})}|\xi_{a_n+\phi(n)-1}]}|\xi_{0,a_n+\phi(n)-1}]\right]\notag\\
=&E\left[\frac{\Lambda_{a_n,\phi(n)-1}(s,\omega) E[e^{sg_{a_n+\phi(n)}(\xi_{a_n+\phi(n)-1},\xi_{a_n+\phi(n)})}|\xi_{a_n+\phi(n)-1}] }{E[e^{sg_{a_n+\phi(n)}(\xi_{a_n+\phi(n)-1},\xi_{a_n+\phi(n)})}|\xi_{a_n+\phi(n)-1}]}\right]
\notag\\
=&E\Lambda_{a_n,\phi(n)-1}(s,\omega)=\cdots=E\Lambda_{a_n,1}(s,\omega)=1.
\end{align}
For any $\varepsilon>0$, by Markov inequality and Eq. (2.1), we have
\begin{align}&\sum_{n=1}^\infty \mathbb{P}\left[\phi^{-1}(n)\log \Lambda_{a_n,\phi(n)}(s,\omega)\geq
\varepsilon\right]\notag\\=&\sum_{n=1}^\infty\mathbb{P}\left[\Lambda_{a_n,\phi(n)}(s,\omega)\geq
\exp(\phi(n)\varepsilon)\right]\notag\\\leq&
\sum_{n=1}^\infty 1\cdot\exp(-\phi(n)\varepsilon)<\infty.\end{align}
 By Borel-Cantelli lemma and arbitrariness of $\varepsilon$, we have
\bes\limsup_n\frac{1}{\phi(n)}\log
\Lambda_{a_n,\phi(n)}(s,\omega)\leq0\
 \ a.e.\ees
Note that
\begin{align}&\frac{1}{\phi(n)}\log\Lambda_{a_n,\phi(n)}(s,\omega)\notag\\=&\frac{1}{\phi(n)}\sum_{k=a_n+1}^{a_n+\phi(n)}\{sg_k(\xi_{k-1},\xi_{k})-\log{E[e^{s
g_k(\xi_{k-1},\xi_k)}|\xi_{k-1}]}\}.\end{align} By Eqs. (2.7) and
(2.8), we have
\bes\limsup_n\frac{1}{\phi(n)}\sum_{k=a_n+1}^{a_n+\phi(n)}\left\{sg_k(\xi_{k-1},\xi_{k})-\log
E[e^{sg_k(\xi_{k-1},\xi_{k})}|\xi_{k-1}]\right\}\leq0\ \ a.e.\ees
Letting $0<s<\gamma$, dividing both sides of Eq. (2.9) by $s$, we
obtain
\bes\limsup_n\frac{1}{\phi(n)}\sum_{k=a_n+1}^{a_n+\phi(n)}\left\{g_k(\xi_{k-1},\xi_{k})-\frac{1}{s
}\log E[e^{s g_k(\xi_{k-1},\xi_{k})}|\xi_{k-1}]\right\}\leq0\ \
a.e.\ees Using the inequalities $\log x\leq x-1\ (x> 0)$ and $0\leq
e^x-1-x\leq \frac{1}{2}x^2e^{|x|}\ ( x\in \mathbf{R})$, from
Eq. (2.10), we have
\begin{align}&\limsup_n\frac{1}{\phi(n)}\sum_{k=a_n+1}^{a_n+\phi(n)}\{g_k(\xi_{k-1},\xi_{k})-E[g_k(\xi_{k-1},\xi_{k})|\xi_{k-1}]\}\notag\\
\leq&\limsup_n\frac{1}{ \phi(n)}\sum_{k=a_n+1}^{a_n+\phi(n)}\left\{\frac{1}{s
}\log E[e^{s g_k(\xi_{k-1},\xi_{k})}|\xi_{k-1}]-E[g_k(\xi_{k-1},\xi_{k})|\xi_{k-1}]\right\}\notag\\
=&\limsup_n\frac{1}{
\phi(n)}\sum_{k=a_n+1}^{a_n+\phi(n)}\left\{\frac{E[(e^{sg_k(\xi_{k-1},\xi_{k})}-1-sg_k(\xi_{k-1},\xi_{k}))|\xi_{k-1}]}{s}\right\}\notag\\
\leq&\frac{s}{2}\limsup_n\frac{1}{
\phi(n)}\sum_{k=a_n+1}^{a_n+\phi(n)}E[g_k^2(\xi_{k-1},\xi_{k})e^{s|g_k(\xi_{k-1},\xi_{k})|}|\xi_{k-1}]\notag\\ \leq &\frac{1}{2}sc(\gamma;\omega)<\infty
\ \ \ a.e.\end{align}
 Letting $s\downarrow0$ in Eq. (2.11), we obtain
\bes\limsup_n\frac{1}{\phi(n)}\sum_{k=a_n+1}^{a_n+\phi(n)}[g_k(\xi_{k-1},\xi_{k})-E(g_k(\xi_{k-1},\xi_{k})|\xi_{k-1})]\leq0\
\ a.e.\ees Letting  $-\gamma<s<0$ in (2.9), similarly, we obtain that
\bes\liminf_n\frac{1}{\phi(n)}\sum_{k=a_n+1}^{a_n+\phi(n)}[g_k(\xi_{k-1},\xi_{k})-E(g_k(\xi_{k-1},\xi_{k})|\xi_{k-1})]\geq0\
\ a.e.\ees Eq. (2.3) follows immediately from Eqs. (2.12) and
(2.13).
\end{proof}

\

Let $\mathbf{1}_{\{\cdot\}}(\cdot)$ be the indicator function
and $S_{m,n}(j;\omega)$  the number of occurrences of $j$ in the
segment $\xi_{m},\cdots,\xi_{m+n-1}$.
It is easy to see that
$$S_{m,n}(j;\omega)=\sum_{k=m}^{m+n-1}\mathbf{1}_{\{j\}}(\xi_k) ~~ $$

Let $S_{m,n}(i,j;\omega)$ be the number of occurrences of the pair
 $(i,j)$ in the sequence of ordered pairs
$(\xi_{m},\xi_{m+1}),\cdots,$ $(\xi_{m+n-1},\xi_{m+n})$. Then
$$S_{m,n}(i,j;\omega)=\sum_{k=m}^{m+n-1}\mathbf{1}_{\{i\}}(\xi_k)\mathbf{1}_{\{j\}}(\xi_{k+1}) ~~$$

\begin{cor} \ Under the conditions of Lemma 1, let $S_{m,n}(j;\omega)$ be defined as before.
Then \bes
\lim_{n}\frac{1}{\phi(n)}\{S_{a_n,\phi(n)}(j;\omega)-\sum_{k=a_n+1}^{a_n+\phi(n)}
p_k(\xi_{k-1},j)\}=0 \ \ a.e.\ees
\end{cor}
\begin{proof} Let $ g_k(x,y)=\mathbf{1}_{\{j\}}(y)$ in Lemma 1. It is easy to see that $\{ g_k(x,y), k\geq 1\}$ satisfy the Eq. (2.2) of Lemma 1. Noticing that
\begin{align} &\sum_{k=a_n+1}^{a_n+\phi(n)}\{g_k(\xi_{k-1},\xi_{k})-
\sum_{l=1}^bg_k(\xi_{k-1},l)p_k(\xi_{k-1},l)\}\notag\\
=&\sum_{k=a_n+1}^{a_n+\phi(n)}\{\mathbf{1}_{\{j\}}(\xi_k)-\sum_{l=1}^b\mathbf{1}_{\{j\}}(l) p_{k}(\xi_{k-1},l)\}\notag\\
=& S_{a_n,\phi(n)}(j;\omega)+\mathbf{1}_{\{j\}}(\xi_{a_n+\phi(n)})-\mathbf{1}_{\{j\}}(\xi_{a_n})-\sum_{k=a_n+1}^{a_n+\phi(n)} p_{k}(\xi_{k-1},j),
\end{align}
Eq. (2.14) follows from Lemma 1.
\end{proof}

\begin{cor} Under the conditions of Lemma 1, let $S_{m,n}(i,j;\omega)$ be defined as before. Then \bes
\lim_{n}\frac{1}{\phi(n)}\{S_{a_n,\phi(n)}(i,j;\omega)-\sum_{k=a_n+1}^{a_n+\phi(n)}
\mathbf{1}_{\{i\}}(\xi_{k-1})p_k(i,j)\}=0 \ \ a.e.\ees
\end{cor}
\begin{proof} Let $g_k(x,y)=\mathbf{1}_{\{i\}}(x)\mathbf{1}_{\{j\}}(y)$ in Lemma 1. It is easy to see that $\{ g_k(x,y), k\geq 1\}$ satisfy the Eq. (2.2) of Lemma 1. Noticing that
\begin{align} &\sum_{k=a_n+1}^{a_n+\phi(n)}\{g_k(\xi_{k-1},\xi_{k})-
\sum_{l=1}^bg_k(\xi_{k-1},l)p_k(\xi_{k-1},l)\}\notag\\
=&\sum_{k=a_n+1}^{a_n+\phi(n)}\{\mathbf{1}_{\{i\}}(\xi_{k-1})\mathbf{1}_{\{j\}}(\xi_k)-\sum_{l=1}^b\mathbf{1}_{\{i\}}(\xi_{k-1})\mathbf{1}_{\{j\}}(l)p_k(\xi_{k-1},l)\}\notag\\
=& S_{a_n,\phi(n)}(i,j;\omega)-\sum_{k=a_n+1}^{a_n+\phi(n)}
\mathbf{1}_{\{i\}}(\xi_{k-1})p_k(i,j),\end{align} Eq. (2.16)
follows from Lemma 1.
\end{proof}

\begin{lem} Let $(a_n)_{n=0}^\infty$ and $(\phi(n))_{n=0}^\infty$ be as in Lemma 1, and $h(x)$ be a bounded function defined on an interval I,
and $(x_n)_{n=0}^\infty$ a sequence in I. If
$$\lim_{n}\frac{1}{\phi(n)}\sum_{k=a_n+1}^{a_n+\phi(n)}|x_{k}-x|=0,$$
and $h(x)$ is continuous at point $x$, then,
$$\lim_{n}\frac{1}{\phi(n)}\sum_{k=a_n+1}^{a_n+\phi(n)}|h(x_{k})-h(x)|=0.$$\end{lem}
\begin{proof} The proof of this lemma is similar to that of Lemma 2 in [12], so we omit it.
\end{proof}

\begin{lem} Let $(\xi_n)_{n=0}^\infty$ be a sequence of arbitrary
random variables taking values from a finite state-space of $\mathbf{X}
=\{1,2,\cdots,b\}$, and let $f_{a_n,\phi(n)}(\omega)$ be defined by
Eq. (1.2). Then $f_{a_n,\phi(n)}(\omega)$ are uniformly
integrable.\end{lem}
\begin{proof} To prove that $f_{a_n,\phi(n)}(\omega)$ are uniformly integrable, it is sufficient to verify the following two conditions (see [5], p.96)

a) For every $\varepsilon>0$, there exists $\delta(\varepsilon)>0$
such that for any $A\in\mathcal{F}$, $$
\mathbb{P}(A)<\delta(\varepsilon)\Longrightarrow
\int_Af_{a_n,\phi(n)}(\omega)d\mathbb{P}<\varepsilon\ \ \textrm{for
every}\  n.$$

b) $Ef_{a_n,\phi(n)}(\omega)$ are bounded for all $n$.

\

Let $A\in\mathcal{F}$. It is easy to see that
\begin{align}&\int_Af_{a_n,\phi(n)}(\omega)d\mathbb{P}\notag\\=&-\int_A\frac{1}{\phi(n)}\log p(\xi_{a_n,\phi(n)})d\mathbb{P}\notag\\
=&-\sum_{x_{a_n},\cdots,x_{a_n+\phi(n)}}\frac{1}{\phi(n)}\log
p(x_{a_n,\phi(n)})\cdot\mathbb{P}(A\cap\{\xi_{a_n,\phi(n)}=x_{a_n,\phi(n)}\})\notag\\
\leq&-\sum_{x_{a_n},\cdots,x_{a_n+\phi(n)}}\frac{1}{\phi(n)}\log
\mathbb{P}(A\cap\{\xi_{a_n,\phi(n)}=x_{a_n,\phi(n)}\})\cdot\mathbb{P}(A\cap\{\xi_{a_n,\phi(n)}=x_{a_n,\phi(n)}\}).\end{align}
Replacing $\log
\mathbb{P}(A\cap\{\xi_{a_n,\phi(n)}=x_{a_n,\phi(n)}\})$ by $\log
\frac{\mathbb{P}(A)}{b^{\phi(n)+1}}$ in Eq. (2.18) and noting that
$$\sum_{x_{a_n},\cdots,x_{a_n+\phi(n)}}\mathbb{P}(A\cap\{\xi_{a_n,\phi(n)}=x_{a_n,\phi(n)}\})=\mathbb{P}(A)=\sum_{x_{a_n},\cdots,x_{a_n+\phi(n)}}\frac{\mathbb{P}(A)}{b^{\phi(n)+1}},$$
by the entropy inequality $$-\sum_{k=1}^sp_k\log
p_k\leq-\sum_{k=1}^sp_k\log q_k,$$ where $p_k,q_k\geq0,\
k=1,2,\cdots,s$ and $\sum_{k=1}^sp_k=\sum_{k=1}^sq_k$, we have
\begin{align}&\int_Af_{a_n,\phi(n)}(\omega)d\mathbb{P}\notag\\\leq&-\sum_{x_{a_n},\cdots,x_{a_n+\phi(n)}}\frac{1}{\phi(n)}\log
\frac{\mathbb{P}(A)}{b^{\phi(n)+1}}\cdot\mathbb{P}(A\cap\{\xi_{a_n,\phi(n)}=x_{a_n,\phi(n)}\})\notag\\
=&-\frac{1}{\phi(n)}\left(\log
\frac{\mathbb{P}(A)}{b^{\phi(n)+1}}\sum_{x_{a_n},\cdots,x_{a_n+\phi(n)}}\mathbb{P}(A\cap\{\xi_{a_n,\phi(n)}=x_{a_n,\phi(n)}\})\right)\notag\\
=&\left(\frac{\phi(n)+1}{\phi(n)}\log
b-\frac{\log\mathbb{P}(A)}{\phi(n)}\right)\cdot\mathbb{P}(A)\notag\\
\leq&(2\log b-\log\mathbb{P}(A))\mathbb{P}(A).\end{align} Since
$\lim_{x\rightarrow0^+}x(2\log b-\log x)=0$, the left hand side of Eq.
(2.19) is small provided $\mathbb{P}(A)$ is small and a) holds.
Letting $A=\Omega$ in Eq. (2.19), we have $$
Ef_{a_n,\phi(n)}(\omega)=\int f_{a_n,\phi(n)}(\omega)d\mathbb{P}\leq
2\log b.$$ Thus b) holds and the proof of the Lemma 5 is complete.
\end{proof}

 \

\section{\bf The Main Results}

\

In this section, we will establish the strong law of large numbers
for frequencies of occurrence of states and the pairs of states for
delayed sums of nonhomogeneous Markov chains and the generalized
entropy ergodic theorem for the Markov chains.

\
\begin{thm}  Suppose $(\xi_n)_{n=0}^\infty$ is a nonhomogeneous Markov chain taking values from a finite state-space of $\mathbf{X} =\{1,2,\cdots,b\}$ with the initial
distribution (1.4) and the transition matrices (1.5). Let
$(a_n)_{n=0}^\infty$ and $(\phi(n))_{n=0}^\infty$ be as in Lemma 1.
Let $S_{a_n,\phi(n)}(i,\omega)$ and $S_{a_n,\phi(n)}(i,j;\omega)$ be
defined as before, and $f_{a_n,\phi(n)}(\omega)$ be defined by
Eq. (1.6). Let $P=(p(i,j))_{b\times b}$ be another transition
matrix, and assume that $P$ is irreducible. If Eq. (2.1) holds and
\bes\lim_{n}\frac{1}{\phi(n)}\sum_{k=a_n+1}^{a_n+\phi(n)}|p_{k}(i,j)-p(i,j)|=0,
\ \ \forall i,j\in \mathbf{X},\ees then
\begin{align}(i)\ \ &\lim_{n}\frac{1}{\phi(n)}S_{a_n,\phi(n)}(i;\omega)=\pi_{i}
\ \ a.e.\ \ \forall i\in\mathbf{X},\\
(ii)\ \
&\lim_{n}\frac{1}{\phi(n)}S_{a_n,\phi(n)}(i,j;\omega)=\pi_{i}p(i,j)
\ \ a.e.\ \forall i,j \in\mathbf{X},\\
 (iii)\ \ &
\lim_{n}f_{a_n,\phi(n)}(\omega)=-\sum_{i=1}^b\sum_{j=1}^b\pi_{i}p(i,j)\log
p(i,j) \ \ a.e.,\  \  \ \ \ \ \ \
\end{align} where $(\pi_1,\cdots,\pi_b)$ is the unique stationary
distribution determined by the transition matrix $P$.
\end{thm}

\begin{rem} It is easy to see that if
$ \lim_n p_n(i,j)= p(i,j)  \ \ \forall i,j \in\mathbf{X},  $
 then Eq. (3.1) holds. Observe that
 $$ \frac{1}{\phi(n)}\sum_{k=a_n+1}^{a_n+\phi(n)}|p_{k}(i,j)-p(i,j)|
   \leq (1+\frac{a_n}{\phi(n)})\frac{1}{a_n+\phi(n)}\sum_{k=1}^{a_n+\phi(n)}|p_{k}(i,j)-p(i,j)|.$$
   If, in addition, $\{\frac{a_n}{\phi(n)}\}$ is bounded, then Eq. (3.1) follows from the following equation
\bes\lim_n\frac{1}{n}\sum_{k=1}^n|p_k(i,j)-p(i,j)|=0 \ \  \forall
i,j \in\mathbf{X}.  \ees But in general Eq. (3.5) may not imply
(3.1). For example, let
$$ P_1=\left[\begin{array}{cc}\frac{1}{3}&\frac{2}{3}\\
\frac{2}{3}&\frac{1}{3}\end{array}\right] ,  P_2=\left[\begin{array}{cc}\frac{1}{2}&\frac{1}{2}\\
\frac{1}{2}&\frac{1}{2}\end{array}\right].  $$
 Let $(\xi_n)_{n=0}^\infty$ be a nonhomogeneous Markov chain with
 transition matrices
$$ P_n=\left\{
\begin{array}{cc}
P_1, \ \   &\textrm{if} \ 2^k\leq n\leq 2^k+k, k\geq 0,\\
\\
P_2,  \ \   &\textrm{otherwise} .
\end{array}
\right.              $$
Let $P=P_2$. It is easy to see that when $2^k\leq n< 2^{k+1}$, for any $i,j \in\mathbf{X}$
\begin{align}
 &\frac{1}{n}\sum_{l=1}^n|p_l(i,j)-p(i,j)|
 \leq \frac{1}{2^k}\sum_{l=1}^{2^{k+1}-1}|p_l(i,j)-p(i,j)|\notag\\
 & \leq \frac{1+2+3+\cdots+(k+1)}{2^k}\frac{1}{6}=\frac{1}{2^k}\frac{(k+2)(k+1)}{2}\frac{1}{6}\rightarrow 0 \ (k\rightarrow \infty). \notag
\end{align}
So Eq. (3.5) holds. However, if we let $a_n=2^n$ and $\phi(n)=n$,
then
\begin{align}
\frac{1}{\phi(n)}\sum_{k=a_n+1}^{a_n+\phi(n)}|p_{k}(i,j)-p(i,j)|=\frac{1}{n}\sum_{k=2^n+1}^{2^n+n}|p_{k}(i,j)-p(i,j)|=\frac{1}{6},\notag
\end{align}
so Eq. (3.1) does not hold.
\end{rem}

\begin{rem}
From Lemma 3, we know that
$\frac{1}{\phi(n)}S_{a_n,\phi(n)}(i;\omega)$,
$\frac{1}{\phi(n)}S_{a_n,\phi(n)}(i,j;\omega)$ and
$f_{a_n,\phi(n)}(\omega)$  are all uniformly integrable, so Eqs.
(3.2), (3.3) and (3.4)  also hold with $\mathcal{L}_1$ convergence.
\end{rem}

\begin{rem}
The right hand side of Eq. (3.4) is actually the entropy rate of a Markov
chain with the transition matrix $P$.
\end{rem}

\begin{rem}
If we define a statistic as follows:
$$\hat{H}=-\sum_{i=1}^b\sum_{i=1}^b\frac{S_{a_n,\phi(n)}(i;\omega)}{\phi(n)}\frac{S_{a_n,\phi(n)}(i,j;\omega)}{S_{a_n,\phi(n)}(i;\omega)}\log \frac{S_{a_n,\phi(n)}(i,j;\omega)}{S_{a_n,\phi(n)}(i;\omega)}, $$
it is easy to see from Theorem 1 that $\hat{H}$ is a strongly consistent estimate of entropy rate $H$, where
$$H= -\sum_{i=1}^b\sum_{j=1}^b\pi_{i}p(i,j)\log
p(i,j).$$ Putting $a_n=2^n$ and $\phi(n)=n$, under the condition of Eq.
(3.1), we can use information from a segment of
$(\xi_n)_{n=0}^\infty$ to estimate the entropy rate of a
nonhomogeneous Markov chain.
\end{rem}

\

\begin{proof} Proof of (i). It is easy to see
that \bes\sum_{k=a_n+1}^{a_n+\phi(n)}p_{k}(
\xi_{k-1},j)=\sum_{k=a_n+1}^{a_n+\phi(n)}\sum_{i=1}^b\mathbf{1}_{\{i\}}(\xi_{k-1})p_k(i,j),\
\ \forall j\in \mathbf{X},\ees and
\bes\sum_{k=a_n+1}^{a_n+\phi(n)}\sum_{i=1}^b\mathbf{1}_{\{i\}}(\xi_{k-1})p(i,j)=\sum_{i=1}^bS_{n,\phi(n)}(i;\omega)p(i,j),\
\ \forall j\in \mathbf{X}.\ees From (3.1), we have that
\begin{align}&\lim_{n}\left|\frac{1}{\phi(n)}\sum_{k=a_n+1}^{a_n+\phi(n)}\sum_{i=1}^b\mathbf{1}_{\{i\}}(\xi_{k-1})[p_k(i,j)-p(i,j)]\right|\notag\\
\leq&\sum_{i=1}^b\lim_{n}\frac{1}{\phi(n)}\sum_{k=a_n+1}^{a_n+\phi(n)}|p_{k}(i,j)-p(i,j)|=0,
\ \ \forall j\in \mathbf{X}.\end{align} Combining  Eqs. (2.14), (3.6),
(3.7) and (3.8), we have
\begin{align}&\lim_{n}\frac{1}{\phi(n)}[S_{n,\phi(n)}(j;\omega)-\sum_{i=1}^bS_{a_n,\phi(n)}(i;\omega)p(i,j)]\notag\\
=&
\lim_{n}\frac{1}{\phi(n)}\sum_{k=a_n+1}^{a_n+\phi(n)}\sum_{i=1}^b\mathbf{1}_{\{i\}}(\xi_{k-1})[p_k(i,j)-p(i,j)]\notag\\
=&0, \ \ a.e. \ \ \forall j\in \mathbf{X}.\end{align}
 Multiplying the
two sides of Eq. (3.9) by $p(j,k)$, and  adding them together
 for $j=1,2,\ldots,b$, we have
\begin{align}0=&
\lim_{n}\frac{1}{\phi(n)}[\sum_{j=1}^bS_{a_n,\phi(n)}(j;\omega)p(j,k)-\sum_{j=1}^b\sum_{i=1}^bS_{a_n,\phi(n)}(i;\omega)p(i,j)p(j,k)]\notag\\
=&\lim_{n}[\sum_{j=1}^b\frac{1}{\phi(n)}S_{a_n,\phi(n)}(j;\omega)p(j,k)-\frac{1}{\phi(n)}S_{a_n,\phi(n)}(k;\omega)]\notag\\&+\lim_n[\frac{1}{\phi(n)}S_{a_n,\phi(n)}(k;\omega)-
\sum_{j=1}^b\sum_{i=1}^b\frac{1}{\phi(n)}S_{a_n,\phi(n)}(i;\omega)p(i,j)p(j,k)]\notag\\=&\lim_{n}[\frac{1}{\phi(n)}S_{a_n,\phi(n)}(k;\omega)-\sum
_{i=1}^{b}\frac{1}{\phi(n)}S_{a_n,\phi(n)}(i;\omega)p^{(2)}(i,k)]\ \
a.e.,
\end{align} where $p^{(l)}(i, k)$ ($l$ is a positive integer) is
the $l$-step transition probability determined by the transition
matrix $P$. By induction, for all $l\geq1$, we have
\bes\lim_{n}\frac{1}{\phi(n)}[S_{a_n,\phi(n)}(k;\omega)-\sum
_{i=1}^{b}S_{a_n,\phi(n)}(i;\omega)p^{(l)}(i,k)]= 0, \ \ a.e., \ees
and
\bes\lim_{n}[\frac{1}{\phi(n)}S_{a_n,\phi(n)}(k;\omega)-\frac{1}{\phi(n)}\sum
_{i=1}^{b}S_{a_n,\phi(n)}(i;\omega)\frac{1}{m}\sum_{l=1}^mp^{(l)}(i,k)]=
0, \ \ a.e. \ees It is easy to see that
$\sum_{i=1}^bS_{a_n,\phi(n)}(i,\omega)=\phi(n)$, by (3.12), we have
for all $m\geq 1$
\bes\limsup_{n}|\frac{1}{\phi(n)}S_{a_n,\phi(n)}(k;\omega)-\pi_k|\leq
\sum_{i=1}^b|\frac{1}{m}\sum_{l=1}^mp^{(l)}(i,k)-\pi_k|
 \ \ a.e. \ees
Because $P$ is irreducible, so
\bes\lim_m\frac{1}{m}\sum_{l=1}^mp^{(l)}(i,k)=\pi_k,\ \ \forall
i\in\mathbf{X},\ees Eq. (3.2) follows from Eqs. (3.13) and (3.14).

\

Proof of (ii). Observe that
\bes\sum_{k=a_n+1}^{a_n+\phi(n)}\mathbf{1}_{\{i\}}(\xi_{k-1})p(i,j)=S_{a_n,\phi(n)}(i;\omega)p(i,j).
\ees From Eq. (3.1), we have that
\bes\lim_n\frac{1}{\phi(n)}\sum_{k=a_n+1}^{a_n+\phi(n)}\mathbf{1}_{\{i\}}(\xi_{k-1})[p_{k}(i,j)-p(i,j)]=0.\ees
Combining  Eqs. (2.16), (3.15) and (3.16), we have
\begin{align}&\lim_{n}\frac{1}{\phi(n)}[S_{a_n,\phi(n)}(i,j;\omega)-S_{a_n,\phi(n)}(i;\omega)p(i,j)]\notag\\
=&\lim_{n}\frac{1}{\phi(n)}\sum_{k=a_n+1}^{a_n+\phi(n)}\mathbf{1}_{\{i\}}(\xi_{k-1})[p_{k}(i,j)-p(i,j)]=0
\ \ a.e.\end{align} Eq. (3.3) follows from Eqs. (3.2) and (3.17).

\

Proof of (iii).  Since
$Ee^{|\log\mu_{a_n}(\xi_{a_n})|}=\sum_{i=1}^be^{-\log\mu_{a_n}(i)}\mu_{a_n}(i)=b$,
by Markov inequality, for every $\varepsilon>0$, form Eq.
(2.1), we have
\begin{align}\sum_{n=1}^\infty \mathbb{P}\left[\phi(n)^{-1}|\log\mu_{a_n}(\xi_{a_n})|\geq
\varepsilon\right]\leq
b\sum_{n=1}^\infty\exp(-\phi(n)\varepsilon)<\infty.\end{align} By Borel-Cantelli lemma,  we
obtain
\bes\lim_n\frac{1}{\phi(n)}\log\mu_{a_n}(\xi_{a_n})=0\ \ a.e.\ees
 Letting $g_k(x,y)=\log p_k(x,y)$ and $\gamma=\frac{1}{2}$ in Lemma 1, and noticing that
\begin{align}&E[(\log p_k(\xi_{k-1},\xi_k))^2e^{\frac{1}{2}|\log
p_k(\xi_{k-1},\xi_k)|}|\xi_{k-1}]\notag\\
=&\sum_{j=1}^bp_k^{-\frac{1}{2}}(\xi_{k-1},j)\log^2p_k(\xi_{k-1},j)
p_k(\xi_{k-1},j)\notag\\
=&\sum_{j=1}^bp_k^{\frac{1}{2}}(\xi_{k-1},j)\log^2p_k(\xi_{k-1},j)\leq 16be^{-2} ,
 \end{align} it
follows from the Lemma 1 that
\bes\lim_n\frac{1}{\phi(n)}\sum_{k=a_n+1}^{a_n+\phi(n)}\left\{\log
p_k(\xi_{k-1},\xi_k)-\sum_{j=1}^bp_k(\xi_{k-1},j)\log
p_k(\xi_{k-1},j)\right\}=0\ \ a.e.\ees Now
\begin{align}&|\frac{1}{\phi(n)}\sum_{k=a_n+1}^{a_n+\phi(n)}\sum_{j=1}^bp_k(\xi_{k-1},j)\log
p_k(\xi_{k-1},j)-\sum_{i=1}^b\pi_i\sum_{j=1}^bp(i,j)\log
p(i,j)|\notag\\
\leq&|\frac{1}{\phi(n)}\sum_{k=a_n+1}^{a_n+\phi(n)}\sum_{i=1}^b\sum_{j=1}^b\mathbf{1}_{\{i\}}(\xi_{k-1})p_k(i,j)\log
p_k(i,j)\notag\\&-\frac{1}{\phi(n)}\sum_{k=a_n+1}^{a_n+\phi(n)}\sum_{i=1}^b\sum_{j=1}^b\mathbf{1}_{\{i\}}(\xi_{k-1})p(i,j)\log p(i,j)|\notag\\
&+|\frac{1}{\phi(n)}\sum_{k=a_n+1}^{a_n+\phi(n)}\sum_{i=1}^b\sum_{j=1}^b\mathbf{1}_{\{i\}}(\xi_{k-1})p(i,j)\log
p(i,j)-\sum_{i=1}^b\pi_i\sum_{j=1}^bp(i,j)\log
p(i,j)|\notag\\
\leq&\sum_{i=1}^b\sum_{j=1}^b\frac{1}{\phi(n)}\sum_{k=a_n+1}^{a_n+\phi(n)}|p_k(i,j)\log
p_k(i,j)-p(i,j)\log p(i,j)|\notag\\
&+\sum_{i=1}^b\sum_{j=1}^b|p(i,j)\log
p(i,j)||\frac{1}{\phi(n)}\sum_{k=a_n+1}^{a_n+\phi(n)}\mathbf{1}_{\{i\}}(\xi_{k-1})-\pi_i|.\end{align}
By Lemma 2, Eq. (3.1) and the continuity of $h(x)=x\log x$, we have
\bes\lim_n\frac{1}{\phi(n)}\sum_{k=a_n+1}^{a_n+\phi(n)}|p_k(i,j)\log
p_k(i,j)-p(i,j)\log p(i,j)|=0\ \ \forall i,j\in\mathbf{X}.\ees Combining
Eqs. (3.21), (3.22), (3.23) and (3.2), we have
\begin{align}&\lim_n\frac{1}{\phi(n)}\sum_{k=a_n+1}^{a_n+\phi(n)}\log
p_k(\xi_{k-1},\xi_k)=\sum_{i=1}^b\pi_i\sum_{j=1}^bp(i,j)\log p(i,j)\
\ a.e.\end{align} From Eqs. (1.6), (3.19) and (3.24), Eq. (3.4)
follows.
\end{proof}

\
\begin{cor}(see [12]).
Under the conditions of Theorem 1, if Eq. (3.5) holds, then
\begin{align}(i)\ \ &\lim_{n}\frac{1}{n}S_{0,n}(i;\omega)=\pi_{i}
\ \ a.e., \textrm{and} \ \ \mathcal{L}_1 \ \forall i\in\mathbf{X},\\
(ii)\ \
&\lim_{n}\frac{1}{n}S_{0,n}(i,j;\omega)=\pi_{i}p(i,j)
\ \ a.e.,\ \textrm{and} \ \ \mathcal{L}_1 \ \forall i,j \in\mathbf{X},\\
 (iii)\ \ &
\lim_{n}f_{0,n}(\omega)=-\sum_{i=1}^b\sum_{j=1}^b\pi_{i}p(i,j)\log
p(i,j) \ \ a.e.,\ \textrm{and} \ \ \mathcal{L}_1, \  \ \ \ \ \ \
\end{align}
where $(\pi_1,\cdots,\pi_b)$ is the unique stationary
distribution determined by the transition matrix $P$.
\end{cor}

\

\section{\bf Acknowledgment }

The authors are very thankful to Professor Keyue Ding who helped us to improve the English of this paper greatly, and very thankful to the reviewers for their valuable comments.

\

   \
   \
   This paper has been accepted for publication in Journal of Theoretical Probability.
       \end{document}